\documentclass{article}

\usepackage{amsmath,amsthm,amssymb,amsfonts}

\usepackage{verbatim}
\usepackage{graphicx}

\usepackage{stmaryrd} 
\usepackage{hyperref}

\usepackage[colorinlistoftodos]{todonotes}

\newtheorem{thm}{Theorem}

\newtheorem{prop}[thm]{Proposition}

\newtheorem{assert}[thm]{Assertion}

\newtheorem{remarks}[thm]{Remark}

\newtheorem{definition}[thm]{Definition}

\newtheorem{exl}[thm]{Example}

\numberwithin{thm}{section}

\newcommand{\adj}{\leftrightarrow}

\newcommand{\adjeq}{\leftrightarroweq}

\def\Z{{\mathbb Z}}
\def\N{{\mathbb N}}

\begin{document}
\title{Remarks on Fixed Point Assertions in Digital Topology, 12}
\author{Laurence Boxer
\thanks{Department of Computer and Information Sciences, Niagara University, NY 14109, USA
and  \newline
Department of Computer Science and Engineering, State University of New York at Buffalo \newline
email: boxer@niagara.edu
\newline
ORCID: 0000-0001-7905-9643
}
}

\date{ }
\maketitle

\begin{abstract} The topic of fixed points in digital metric spaces has drawn yet more
publications with assertions that are incorrect, incorrectly proven, trivial, or
incoherently stated. We discuss publications with 
bad assertions concerning fixed points of self-functions on digital images, as in some of
our previous papers.


\end{abstract}

\section{Introduction}
{\bf Fixed points in digital topology} - this is a topic that has inspired some beautiful
results, as well as many assertions that 
are incorrect, incorrectly
proven, trivial, or not presented coherently. Here, we continue the work
of~\cite{BxSt19, Bx19, Bx19-3, Bx20, Bx22, BxBad6, BxBad7, BxBad8, BxBad9, BxBad10, BxBad11}
as we discuss several other papers with 
assertions that merit at least one of the 
descriptions in the previous sentence. 
Papers drawing our disapproval in
the current work have all come to our attention since acceptance for publication
of~\cite{BxBad11}.

The study of the {\bf coincidence points}
of two or more functions has often been 
linked to the study of fixed points. Among the
questions we consider are some concerned with
coincidence points.

\section{Preliminaries}
Much of the material in this section is 
quoted or paraphrased from~\cite{BxBad10}.

Notations:
\begin{itemize}
    \item $\Z$ represents the set of integers.
    \item $\N$ represents the set of natural numbers.
    \item $\#X$ is the number of distinct
    members of the set~$X$.
\end{itemize}

\subsection{Adjacencies, connectedness, 
continuity, fixed point}

A {\em digital image} is a pair $(X,\kappa)$,
where $X \subset \Z^n$ for some $n \in \N$ and
$\kappa$ is an adjacency relation on pairs of
members of~$X$.

In a digital image $(X,\kappa)$, if
$x,y \in X$, we use the notation
$x \adj_{\kappa}y$ to
mean $x$ and $y$ are $\kappa$-adjacent; we may write
$x \adj y$ when $\kappa$ can be understood. 
We write $x \adjeq_{\kappa}y$, or $x \adjeq y$
when $\kappa$ can be understood, to
mean 
$x \adj_{\kappa}y$ or $x=y$.

The most commonly used adjacencies in the study of digital images 
are the $c_u$ adjacencies. These are defined as follows.
\begin{definition}
\label{cu-adj-Def}
Let $X \subset \Z^n$. Let $u \in \Z$, $1 \le u \le n$. Let 
$x=(x_1, \ldots, x_n),~y=(y_1,\ldots,y_n) \in X$. Then $x \adj_{c_u} y$ if 
\begin{itemize}
    \item $x \neq y$,
    \item for at most $u$ distinct indices~$i$,
    $|x_i - y_i| = 1$, and
    \item for all indices $j$ such that $|x_j - y_j| \neq 1$ we have $x_j=y_j$.
\end{itemize}
\end{definition}

\begin{definition}
\label{path}
{\rm (See \cite{Khalimsky})} 
    Let $(X,\kappa)$ be a digital image. Let
    $x,y \in X$. Suppose there is a set
    $P = \{x_i\}_{i=0}^n \subset X$ such that
$x=x_0$, $x_i \adj_{\kappa} x_{i+1}$ for
$0 \le i < n$, and $x_n=y$. Then $P$ is a
{\em $\kappa$-path} (or just a {\em path}
when $\kappa$ is understood) in $X$ from $x$ to $y$,
and $n$ is the {\em length} of this path.
\end{definition}

\begin{definition}
{\rm \cite{Rosenfeld}}
A digital image $(X,\kappa)$ is
{\em $\kappa$-connected}, or just {\em connected} when
$\kappa$ is understood, if given $x,y \in X$ there
is a $\kappa$-path in $X$ from $x$ to $y$. The {\rm $\kappa$-component of~$x$ in~$X$} is the
maximal $\kappa$-connected subset
of~$X$ containing~$x$.
\end{definition}

\begin{definition}
{\rm \cite{Rosenfeld, Bx99}}
Let $(X,\kappa)$ and $(Y,\lambda)$ be digital
images. A function $f: X \to Y$ is 
{\em $(\kappa,\lambda)$-continuous}, or
{\em $\kappa$-continuous} if $(X,\kappa)=(Y,\lambda)$, or
{\em digitally continuous} when $\kappa$ and
$\lambda$ are understood, if for every
$\kappa$-connected subset $X'$ of $X$,
$f(X')$ is a $\lambda$-connected subset of $Y$.
\end{definition}

\begin{thm}
{\rm \cite{Bx99}}
A function $f: X \to Y$ between digital images
$(X,\kappa)$ and $(Y,\lambda)$ is
$(\kappa,\lambda)$-continuous if and only if for
every $x,y \in X$, if $x \adj_{\kappa} y$ then
$f(x) \adjeq_{\lambda} f(y)$.
\end{thm}

A {\em fixed point} of a function $f: X \to X$ 
is a point $x \in X$ such that $f(x) = x$. 

As a convenience, if $x$ is a point in the domain of a function $f$, we will often
abbreviate ``$f(x)$" as ``$fx$".

\subsection{Digital metric spaces}
\label{DigMetSp}
A {\em digital metric space}~\cite{EgeKaraca-Ban} is a triple
$(X,d,\kappa)$, where $(X,\kappa)$ is a digital image and $d$ is a metric on $X$. The
metric is usually taken to be the Euclidean
metric or some other $\ell_p$ metric; 
alternately, $d$ might be taken to be the
shortest path metric. These are defined
as follows.
\begin{itemize}
    \item Given 
          $x = (x_1, \ldots, x_n) \in \Z^n$,
          $y = (y_1, \ldots, y_n) \in \Z^n$,
          $p > 0$, $d$ is the $\ell_p$ metric
          if \[ d(x,y) =
          \left ( \sum_{i=1}^n
          \mid x_i - y_i \mid ^ p
          \right ) ^ {1/p}. \]
          Note the special cases: if $p=1$ we
          have the {\em Manhattan metric}; if
          $p=2$ we have the 
          {\em Euclidean metric}.
    \item \cite{ChartTian} If $(X,\kappa)$ is a 
          connected digital image, 
          $d$ is the {\em shortest path metric}
          if for $x,y \in X$, $d(x,y)$ is the 
          length of a shortest
          $\kappa$-path in $X$ from $x$ to $y$.
\end{itemize}


We say a metric space $(X,d)$ is {\em uniformly discrete}
if there exists $\varepsilon > 0$ such that
$x,y \in X$ and $d(x,y) < \varepsilon$ implies $x=y$.

\begin{remarks}
\label{unifDiscrete}
If $X$ is finite or  
\begin{itemize}
\item {\rm \cite{Bx19-3}}
$d$ is an $\ell_p$ metric, or
\item $(X,\kappa)$ is connected and $d$ is 
the shortest path metric,
\end{itemize}
then $(X,d)$ is uniformly discrete.

For an example of a digital metric space
that is not uniformly discrete, see
Example~2.10 of~{\rm \cite{Bx20}}.
\end{remarks}

We say a sequence $\{x_n\}_{n=0}^{\infty}$ is 
{\em eventually constant} if for some $m>0$, 
$n>m$ implies $x_n=x_m$.
The notions of convergent sequence and complete digital metric space are often trivial, 
e.g., if the digital image is uniformly 
discrete, as noted in the following, a minor 
generalization of results 
of~\cite{HanBan,BxSt19}.

\begin{prop}
\label{eventuallyConst}
{\rm \cite{Bx20}}
If $(X,d)$ is a uniformly discrete metric space,
then any Cauchy sequence in $X$
is eventually constant, and $(X,d)$ is a complete metric space.
\end{prop}

\subsection{On fixed points and coincidence 
points}
A function $f: X \to X$ has a {\em fixed point} $x_0 \in X$ if
$fx_0 = x_0$. If $X$ is a topological space or if $(X,\kappa)$ is
a digital image, $X$ or, respectively, $(X,\kappa)$ has the
{\em fixed point property} (FPP) if for every continuous
(respectively, $\kappa$-continuous) $f: X \to X$ has a fixed point.
But the FPP turns out to be trivial in digital topology, as shown by the following.

\begin{thm}
\label{BEKLLthm}
    {\rm \cite{BEKLL}} A digital image $(X,\kappa)$ has the FPP if
    and only if $\#X = 1$.
\end{thm}

However, the study of conditions that 
lead to the existence of fixed points 
has been a fruitful topic of research.

Given functions $f,g: X \to Y$, we say
$x_0 \in X$ is a {\em coincidence point 
of}~$f$ and~$g$ if $f(x_0) = g(x_0)$.
Perhaps because, when $Y=X$, a common fixed point
of~$f$ and~$g$ is a coincidence point,
coincidence points and fixed points are
sometimes studied together.
For more on the existence of,
and properties of, coincidence points 
in digital topology, see \cite{AbdullahiEtal20}.

\subsection{The digital Banach contraction principle}

We have the following.
\begin{definition}
    Let $(X,d,\kappa)$ be a digital metric
    space. Let $f: X \to X$ satisfy, for
    all $x,y \in X$ and some $c$ satisfying
    $0 < c < 1$,
    \[ d(fx,fy) < c \cdot d(x,y). \]
    Then $f$ is a {\em digital contraction}.
\end{definition}

The following is sometimes called
the digital Banach contraction principle.
\begin{thm}
\label{EgeKaracaAssertion}
    {\rm \cite{EgeKaraca-Ban}} A digital
    contraction has a unique fixed point.
\end{thm}

The proof of Theorem~\ref{EgeKaracaAssertion}
in~\cite{EgeKaraca-Ban}
was shown to be incorrect in~\cite{BxBad10}, 
A correct proof, with the additional 
requirement that the metric be uniformly
discrete, was given in~\cite{BxBad10}.

\section{\cite{MishraEtal}'s alleged coincidence point}
\begin{figure}
    \centering
    \includegraphics[width=5in]{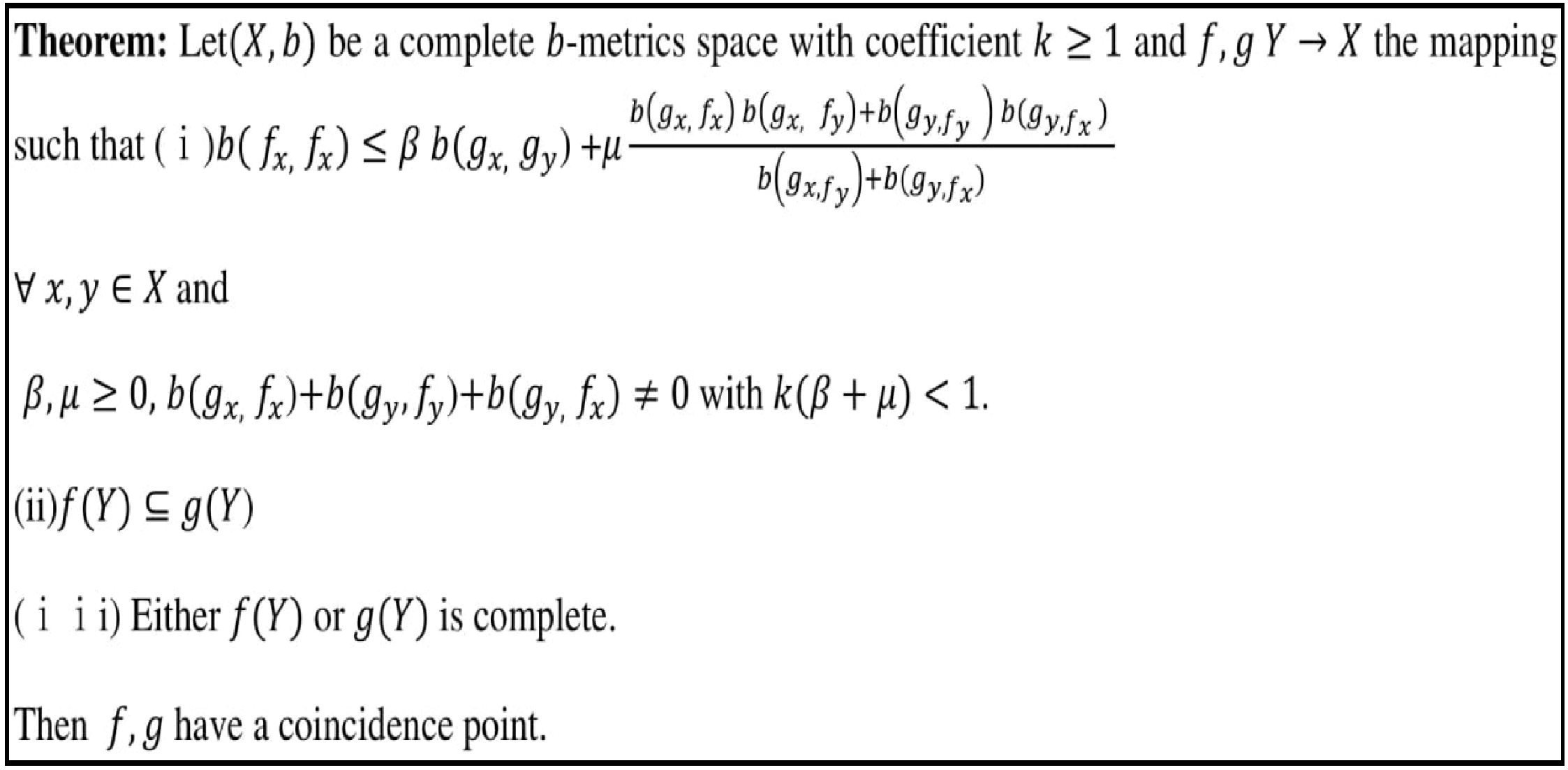}
    \caption{The "theorem" of~\cite{MishraEtal}}
    \label{fig:MishraEtalThm}
\end{figure}

If $f,g: X \to Y$ and $x_0 \in X$ is such
that $f(x_0) = g(x_0)$, then $x_0$ is
a {\em coincidence point} of~$f$ and~$g$.
The study of coincidence points is often
related to the study of fixed points.

The only assertion presented as new 
in~\cite{MishraEtal} is shown
in Figure~\ref{fig:MishraEtalThm}. Note
the following.
\begin{itemize}
    \item The left side of inequality (i) is
        $b(fx, fx) = 0$. This suggests that
        the inequality will be true whenever
        the right side is defined.
    \item The hypotheses of the assertion 
    include that the right side of 
    inequality~(i) is valid for all 
    $x,y \in Y$. The assertion is that there
          is a coincidence point, but suppose
          $x_0 \in Y$ is a coincidence point
          of~$f$ and~$g$. Then the right
          side of inequality~(i), when 
          evaluated for $x=y=x_0$, has a 
          denominator of~$0$ in its second
          term. Thus the very assertion that
          the authors seek to establish
          yields an undefined statement.
          
          We conclude that 
          therefore~\cite{MishraEtal}
          should have been rejected or 
          required to undergo major revision.
    \item The paper is further flawed, e.g.,
          by introduction of undefined points $z_n,~z_{n+1}$ in the third line
          of the ``proof" of the assertion,
          and by a Conclusion section that
          seems composed of random phrases.   
\end{itemize}

Indeed, the assertion of
Figure~\ref{fig:MishraEtalThm} is false, even
if modified so that inequality~(i) is
assumed true only when its statement is
defined. Failure of the assertion is shown in 
the following.

\begin{exl}
Let $Y$ be a digital simple closed curve 
of~$n \ge 4$ points $\{x_i\}_{i=0}^{n-1}$ 
indexed circularly. Let $f$ and $g$ be the
self maps on~$X$ given by
\[ f(x_i)= x_{(i+1) \mod n}~~~\mbox{\rm and}~~~
g(x_i) = x_{(i+2) \mod n}.
\]
\end{exl}
Clearly, $f$ and $g$ do not have a coincidence
point. Further, if $gx_i = fx_j$ then 
$(i+2) \mod n = (j+1) \mod n$, or 
$(i+1) \mod n =j$; while if $gx_j = fx_i$
then $(j+2) \mod n = (i+1) \mod n$, or
$j = (i-1) \mod n$. Therefore, the second
term of the right side of~(i)
never has a denominator of~$0$. Thus, 
the left side of~(i) is $0$
and the right side of~(i)
is nonnegative, for all $x,y \in Y$,
so~(i) is satisfied. Further, we have
$f(Y)=Y=g(Y)$ and, by 
Proposition~\ref{eventuallyConst},
$(Y,b)$ is complete. Thus, the hypotheses
of the assertion are satisfied, and its
conclusion is false.

\section{\cite{OzEtal}'s alleged common fixed
points}
\cite{OzEtal} is another much-flawed paper.

\subsection{\cite{OzEtal}'s ``Theorem" 4}
We have the following, paraphrased from the
original to simplify notation and presentation.

\begin{definition}
    {\rm \cite{OzEtal}} 
    \label{rational}
    Let $(X,d, \kappa)$ be a digital 
    metric space. Let $S,T: X \to X$. Suppose
    there exist positive
    $\alpha, \beta, \gamma, \theta, \phi$
    such that 
    $\alpha + \beta + \gamma + \theta + \phi < 1$ and, for all $c,y \in X$,
    \[
        d(Sc,Ty) \le \alpha d(c,y) +
        \beta \frac{d(c,Sc) d(y,Ty)}{d(c,y)}
    + \gamma \frac{d(y,Tc) d(c,Ty)}{d(c,y)}
    \]
    \begin{equation}
    \label{OzEtal4.1ineq}
    + \theta \frac{d(c,Sc) d(y,Ty)}{d(c,Ty) + d(c,y) + d(y,Tc)} 
    + \phi \frac{d(y,Ty) d(c,Ty)}{d(c,Sy)+d(y,Ty)}
    \end{equation}     
    for all $c,y \in X$ such that $c \neq y$.
    Then  $(S,T)$ is a pair of {\em rational 
    type contraction mappings}.
\end{definition}

The following is stated as ``Theorem" 4
of~\cite{OzEtal}.

\begin{assert}
\label{OzEtal4}
    Let $(X,d,\kappa)$ be a digital metric
    space. Let $S,T: X \to X$. Suppose
    \begin{itemize}
        \item $T(X) \subset S(X)$,
        \item $S$ is $\kappa$-continuous, and
        \item $(S,T)$ is a pair of new type
        rational contractions.
        \item $S$ and $T$ commute with each
        other.
    \end{itemize}
    Then $S$ and $T$ have a unique common
    fixed point.
\end{assert}

Note it is not explained in~\cite{OzEtal} 
if ``rational type mappings" are the same as
``new type rational contractions". We assume
they are the same in the following discussion.

The argument in~\cite{OzEtal} offered as proof
of this assertion assumes incorrectly that a
digital metric space is complete; that
a digital metric space can fail to be
complete is shown in
Remark~\ref{nonStdRem} below.

\begin{exl}
{\rm \cite{Bx20}}
\label{nonStdMetric}
Let $X = \N \cup \{0\}$, 
\[ d(x,y) = \left \{ \begin{array}{ll}
            0 & \mbox{if } x=0=y; \\
            1/x & \mbox{if } x \neq 0 = y; \\
            1/y & \mbox{if } x = 0 \neq y; \\
            |1/x - 1/y| & \mbox{if } x \neq 0 \neq y.
            \end{array} \right .
\]
Then $d$ is a metric, and $\lim_{n \to \infty} d(n,0) = 0$. However,
the sequence $x_n=n+1$ satisfies 
\[ \lim_{n \to \infty} d(0, x_n) = 0 \neq c_0.
\]
\end{exl}

\begin{remarks}
\label{nonStdRem}
    If $(X,d)$ are as in Example~\ref{nonStdMetric},
    it follows that $(X \setminus \{0\},d,c_1)
    = (\N,d,c_1)$
    is a digital metric space that is 
    not complete.
\end{remarks}

Perhaps a more serious error in the ``proof"
given in~\cite{OzEtal} is discussed as
follows. The ``proof" starts:
Let $c_0 \in X$ and, inductively,
    \begin{equation}
    \label{S-onEven,T-onOdd}
        c_{2j+1} = Sc_{2j},~~~~ c_{2j+2} = Tc_{2j+1}.
    \end{equation}
\begin{remarks}
\label{unsimplifiable}
    However, there is no simplification given,
either by hypothesis or by derivation,
for the result of applying~$S$ to a member of
the sequence that has an odd index, or
for the result of applying~$T$ to a member of
the sequence that has an even index.
\end{remarks}

The authors correctly derive that there is a
constant~$\lambda$, $0 < \lambda < 1$, such that
    \begin{equation} 
    \label{oddStartBdd}
    d(c_{2j+1},c_{2j+2}) \le
       \lambda d(c_{2j}, c_{2j+1}).
    \end{equation}
They proceed to claim that similarly,
\[  d(c_{2j+2},c_{2j+3}) \le
       \lambda d(c_{2j+1}, c_{2j+2}).
\]
But the latter claim does not follow. An
attempt to mimic the argument for~(\ref{oddStartBdd}) fails, due to
the observation of Remark~\ref{unsimplifiable}.
This is seen as follows. 
\[ d(c_{2j+2}, c_{2j+3}) = d(Tc_{2j+1}, Sc_{2j+2})
   = d(Sc_{2j+2}, Tc_{2j+1}).
\]
Inequality~(\ref{OzEtal4.1ineq}) implies
that the latter expression is bounded above
by an expression in which 
\[ \begin{array}{cll}
\underline{\rm the~term~with~coefficient} & ~~~~~ & \underline{\rm has~the~unsimplifiable~subexpression} \\
\gamma & & Tc_{2j+2} \\
\theta & & Tc_{2j+2} \\
\phi & & Sc_{2j+3}
\end{array}
\]

We conclude that Assertion~\ref{OzEtal4}
is unproven.

\subsection{\cite{OzEtal}'s Example 1}
Example 1 of \cite{OzEtal} claims that the
functions $S,T: ([0,1]_{\Z}, d) \to ([0,1]_{\Z}, d)$
given by
\[ S(c) = (2/3)c^2 + 1/3, ~~~~~ T(c) = (2/5)c^2 + 1/5,
\]
satisfy the hypotheses of Assertion~\ref{OzEtal4},
where $d(y,c) = |y-c|$. This claim is false, 
since~$S$ and~$T$ are not integer-valued.

\subsection{\cite{OzEtal}'s ``corollaries"}
\begin{figure}
    \centering
    \includegraphics[width=5in]{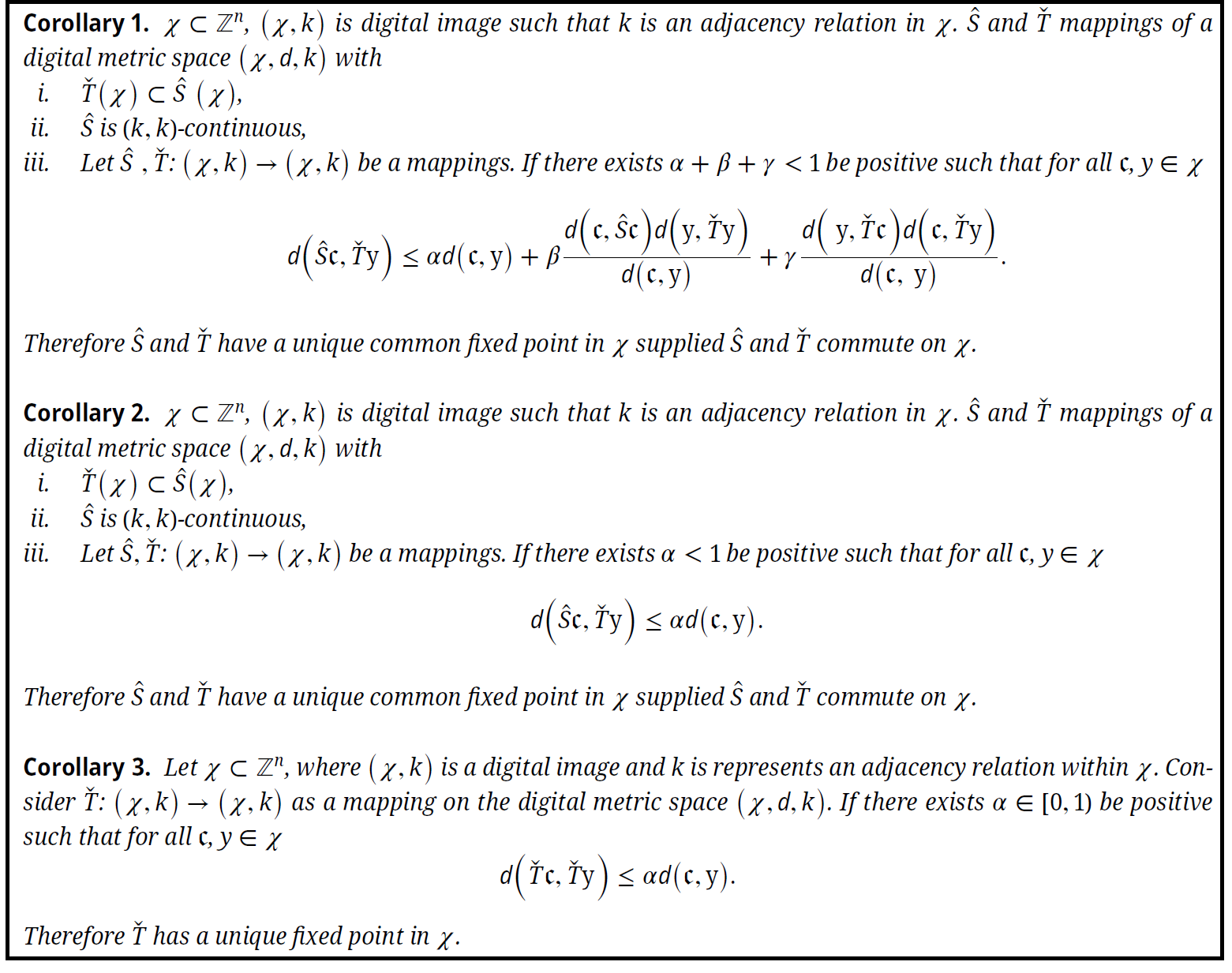}
    \caption{`Corollaries" 1, 2, 3 of \cite{OzEtal}}
    \label{fig:OzEtalCors}
\end{figure}

``Corollaries" 1, 2, and 3 of~\cite{OzEtal} 
(see Figure~\ref{fig:OzEtalCors}) are
intended to follow from the discredited
Assertion~\ref{OzEtal4}; therefore, these
assertions are also unproven. 

By using some of the assumptions 
of ``Corollary"~2, we can obtain a
coincidence-points result. Notice in the
following that we do not require that
$Tx \subset S(Y)$, that $S$ is continuous, or that~$S$ and~$T$ commute; however, we 
require~$d$ to be uniformly discrete.

\begin{thm}
    Let $(X,d)$ be a metric space and
    $S,T: X \to X$. Suppose~$d$ is
    uniformly discrete. Suppose
    for some $\alpha$ such that
              $0 \le \alpha < 1$ and
              all $c,y \in X$,
              \[ d(Sc,Ty) \le \alpha d(c,y).\]
    Then $S$ and $T$ have at least one
    coincidence point.
\end{thm}

\begin{proof}
The assertion is trivial if $\#X = 1$, so
assume $\#X > 1$.

    Let $c_0,y_0 \in X$. For $n>0$ let
    $c_{n+1} = Sc_n$ and $y_{n+1} = Ty_n$. We
    have
    \[ d(c_{n+1},y_{n+1}) =
       d(Sc_n, Ty_n) \le
       \alpha d(c_n, y_n).
    \]
    From this, an easy induction leads to
    \[ d(c_{n+1},y_{n+1}) \le
    \alpha^{n+1} d(c_0, y_0) 
    \to_{n \to \infty} 0.
    \]
    Since $d$ is uniformly discrete, it
    follows that there exists $m > 0$ such 
    that $n \ge m$ implies $d(c_n,y_n) = 0$,
    or $c_n = y_n$. Then
    \[ d(Sc_n,Tc_n) = d(Sc_n, Ty_n) \le
      \alpha d(c_n,y_n) = 0
    \]
    so $c_n$ is a
    coincidence point for~$S$ and~$T$.
\end{proof}

\section{\cite{SinghEtal}'s alleged fixed points}
\subsection{Dubious assertions}
In this section, we mention dubious assertions 
that \cite{SinghEtal} cites. These include:
\begin{itemize}
    \item The assertion of~\cite{EgeKaraca-Ban}
    that a digital contraction map is digitally
    continuous. This assertion is disproved
    in Example~4.1 of~\cite{BxSt19}.
    \item ``Theorem" 3.1 is stated
    without an implication.
\end{itemize}

\subsection{\cite{SinghEtal}'s ``Theorem" 4.1}
\begin{figure}
    \centering
    \includegraphics[width=5in]{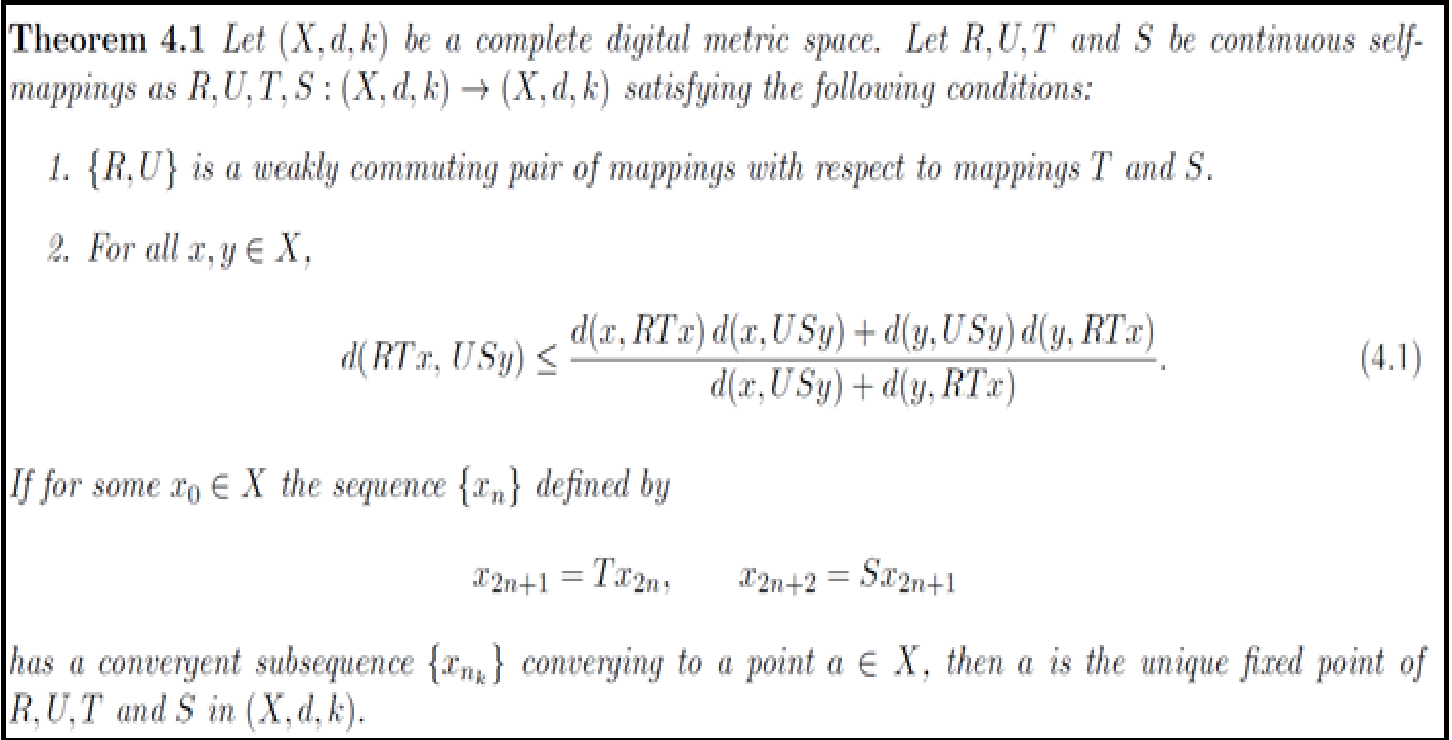}
    \caption{``Theorem" 4.1 of \cite{SinghEtal}}
    \label{fig:SinghEtal4.1}
\end{figure}

``Theorem" 4.1 of \cite{SinghEtal} is shown
in Figure~\ref{fig:SinghEtal4.1}. The
following deficiencies are present.
\begin{itemize}
    \item Although the paper has defined a
    ``weakly commuting pair of mappings" at
    its Definition~2.3, it has not defined 
    what the assertion calls a
    ``weakly commuting pair of mappings
    with respect to mappings $T$ and $S$".
   \item If $a$ is a common fixed point of
         $R,U,T,S$, the denominator in the
         right side of (4.1) is equal to 0
         when $x=y=a$,
         so (4.1) is undefined in this case.
         This contradicts the hypothesis that
         (4.1) holds for all $x,y \in X$.
\end{itemize}

Perhaps the most important error in this 
assertion is found in its ``proof". On
page~5 of~\cite{SinghEtal}, the authors
assume that $a \in X$ is a fixed point
of~$R \circ T$. They derive the trivial inequality
\[ d(a,RTa) \le d(a,RTa),
\]
and claim this inequality contradicts
the assumption that $a$ is a fixed point
of~$R \circ T$; clearly, it does not.

We must conclude that “Theorem” 4.1 is not
proven.

\subsection{``Corollaries" 4.1 and 4.2 of \cite{SinghEtal}}
\begin{figure}
    \centering
    \includegraphics[width=5in]{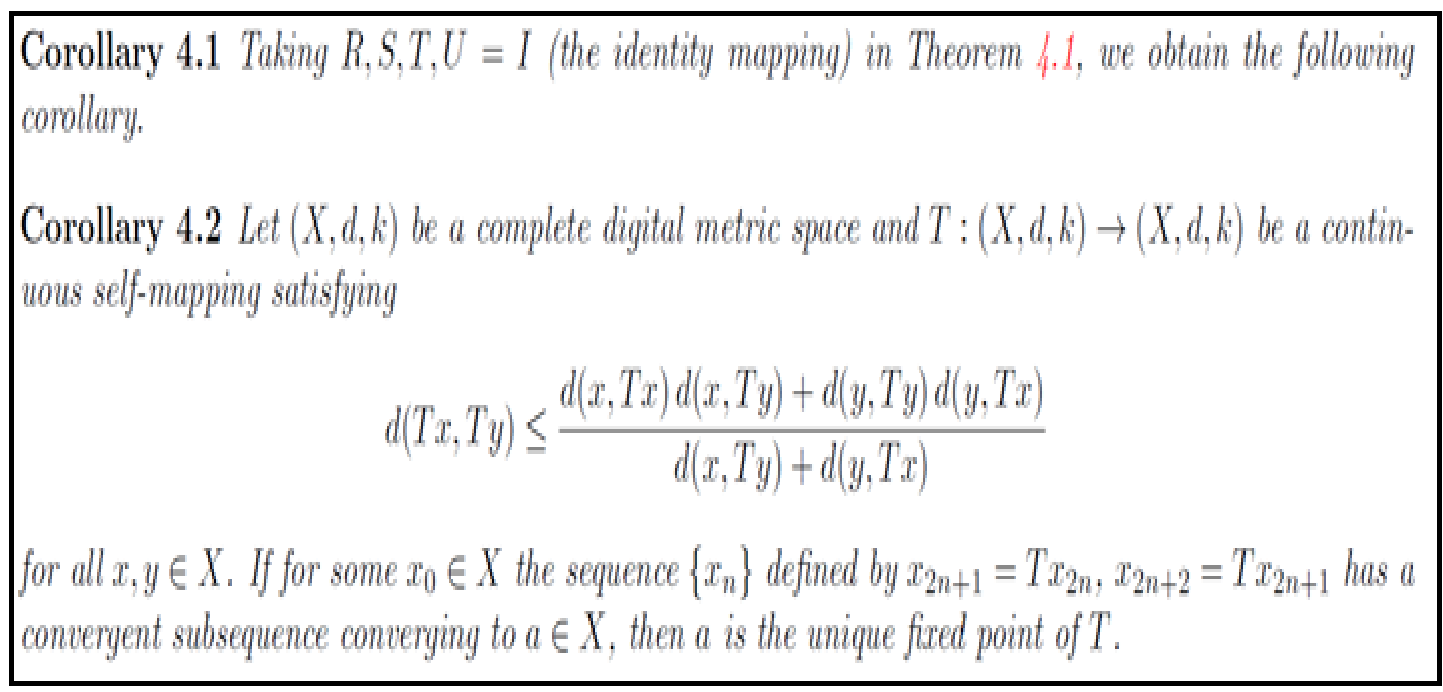}
    \caption{``Corollaries" 4.1 and 4.2 of \cite{SinghEtal}}
    \label{fig:SinghEtalCors4-1-4-2}
\end{figure}

``Corollaries" 4.1 and 4.2 of \cite{SinghEtal}
are shown in Figure~\ref{fig:SinghEtalCors4-1-4-2}. Since they are stated as corollaries of
the now-discredited ``Theorem"~4.1, 
``Corollaries" 4.1 and 4.2 of \cite{SinghEtal}
are now also discredited.

Since ``Corollary" 4.1 takes $R,S,T,U$ to
be identity maps and the map~$T$ of
``Corollary" 4.2 is not required to be
an identity map, it seems ``the following
corollary" of~4.1 does not refer to 4.2. Thus
we conclude 4.1 is not completely stated.

``Corollary" 4.2 states that its inequality
must hold for all $x,y \in X$. However, for
$x=a=y$ where~$a$ is a fixed point of~$T$,
the denominator of the right side of the
inequality is 0, hence the inequality is
undefined.

\subsection{Example 1 of \cite{SinghEtal}}
Example 1 of \cite{SinghEtal} has 
$X=\{3,4,5\}$ and $d$ as the Euclidean metric.
\begin{itemize}
    \item At the bottom of page~6, we see in the distance matrix~$D$ the claim that 
    $d(5,2) = 1$; 
clearly, this should be $d(5,4) = 1$. 
\item In the 3rd bullet of page~7, we have
the claim that 0 and 2 are non-adjacent.
However, $0 \not \in X$.
\end{itemize}

\subsection{Example 2 of \cite{SinghEtal}}
In Example 2 of \cite{SinghEtal},
$X = [0,9]_{\Z}$, and it is claimed that
$R,U,T,S: X \to X$, where
\[ R(x)=x-1,~~~~U(x)=2x,~~~~S(x)=x+1,~~~
T(x)= \lfloor x/2 \rfloor.\]
But this claim is false, since 
$R(0) \not \in X$, 
$U([5,9]_{\Z}) \cap X = \emptyset$, and
$S(9) \not \in X$.

\subsection{``Theorem" 4.2 of \cite{SinghEtal}}
\begin{figure}
    \centering
    \includegraphics[width=5in]{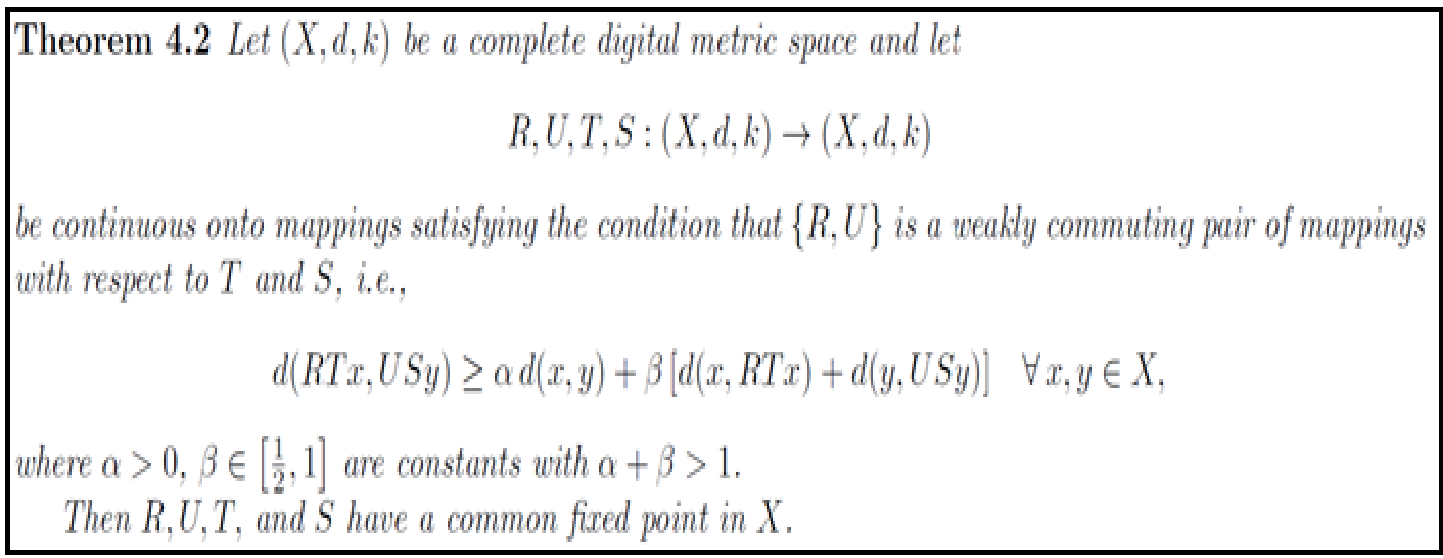}
    \caption{``Theorem" 4.2 of \cite{SinghEtal}}
    \label{fig:SinghEtal4.2}
\end{figure}

This ``Theorem" is shown in
Figure~\ref{fig:SinghEtal4.2}. In their
attempt to prove it, the authors reach,
near the end of their argument,
the existence of a common fixed point~$a$
of $R \circ T$ and $U \circ S$. They then
claim, with no justification, that~$a$
is a common fixed point of $R,U,T,S$. We must
conclude that this assertion is unproven.

\subsection{Example 3 of \cite{SinghEtal}
}
In this example, $X = \{3,4,5\}$ and
$T(x)=S(x)=R(x)=U(x) = 1$ for all
$x \in X$. The example is flawed by the following.
\begin{itemize}
    \item None of $T,S,R,U$ map $X$ into $X$.
    \item Were $1 \in X$, it would be
    obvious that $1$ is the only common
    fixed point of functions defined as above.
    \item The authors conclude that 3 is
    the common fixed point, which is false
    for a constant function that takes the
    value~1.
\end{itemize}

\section{Further remarks}
We paraphrase~\cite{BxBad8}:
\begin{quote}
We have discussed several papers that seek to advance
fixed point assertions for digital metric spaces.
Most of their assertions are incorrect or incorrectly proven. 
The authors, the referees who
approved their publication, and, perhaps, 
``predatory journals"
that accept payment for publication regardless of quality - share responsibility for
deficient ``mathematics".
\end{quote}

\end{document}